\documentclass[11pt]{article}

\usepackage[a4paper,margin=1.15in]{geometry}
\usepackage{amsmath,amssymb,amsthm,mathtools}
\usepackage{microtype}
\usepackage[hidelinks]{hyperref}

\newtheorem{theorem}{Theorem}[section]
\newtheorem{proposition}[theorem]{Proposition}
\newtheorem{lemma}[theorem]{Lemma}

\theoremstyle{remark}

\newcommand{\R}{\mathbb R}
\newcommand{\C}{\mathbb C}
\newcommand{\E}{\mathbb E}
\newcommand{\dd}{\,\mathrm d}

\newcommand{\dist}{\operatorname{dist}}
\newcommand{\rf}[2]{\left(#1\right)_{#2}}

\title{A new hyperbolicity wedge and a joint semicircle limit for Jensen polynomials of Riemann's \texorpdfstring{$\xi$}{xi}-function}
\author{Jonathan Holland}
\date{3 August 2026}

\begin{document}

\maketitle

\begin{abstract}
Let
\[
 \xi\!\left(\frac12+z\right)
   =\sum_{n\geq 0}\frac{\gamma(n)}{n!}z^{2n},
 \qquad
 J^{d,n}(X)
   =\sum_{j=0}^{d}\binom dj\gamma(n+j)X^j .
\]
The Riemann hypothesis is equivalent to the hyperbolicity of
$J^{d,n}$ for every $d,n\geq0$.  We prove that there is an absolute
constant $K>0$ such that
\[
 n^3\log^2(n+2)\geq Kd^5
 \quad\Longrightarrow\quad
 J^{d,n}\ \text{is hyperbolic}.
\]
Along every sequence with $n,d\to\infty$ in this region, the
empirical measure of the naturally centered and scaled zeros also
converges to Wigner's semicircle law.  This gives a simultaneous
degree--derivative version of the global semicircle consequence of
the fixed-degree Hermite limit of Griffin, Ono, Rolen, and Zagier.
\end{abstract}

\section{Introduction and statement of the result}

Riemann's completed zeta-function is
\[
 \xi(s)
 =\frac12s(s-1)\pi^{-s/2}\Gamma\!\left(\frac{s}{2}\right)\zeta(s).
\]
Its functional equation implies that
\[
 \xi_{\mathrm c}(z):=\xi\!\left(\frac12+z\right)
 =\sum_{n=0}^{\infty}\frac{\gamma(n)}{n!}z^{2n}
\]
is an even real entire function.  The coefficients satisfy
$\gamma(n)>0$.  For integers $d,n\geq0$, define
\[
 J^{d,n}(X)
 :=\sum_{j=0}^{d}\binom dj\gamma(n+j)X^j.
\]
A real polynomial is called \emph{hyperbolic} if all its zeros are
real.  P\'olya's Jensen criterion \cite{Polya} says that the Riemann
hypothesis is equivalent to the hyperbolicity of $J^{d,n}$ for all
$d,n$.

Our main result is the following unconditional region in the
$(d,n)$-plane.

\begin{theorem}[Main theorem]\label{thm:main}
There is an absolute constant $K>0$ such that, for all integers
$d\geq1$ and $n\geq0$ satisfying
\[
 n^3\log^2(n+2)\geq Kd^5,
\]
the Jensen polynomial $J^{d,n}$ has $d$ distinct negative real
zeros.
\end{theorem}

Thus hyperbolicity holds uniformly for
\[
 d\leq c\,n^{3/5}\log^{2/5}(n+2)
\]
with a suitable absolute $c>0$.

There is also a consequence for the derivative-aspect random-matrix
picture.  Griffin, Ono, Rolen, and Zagier proved that, for each fixed
$d$, a centered and scaled Jensen polynomial tends to a Hermite
polynomial as $n\to\infty$ \cite[Theorem~3]{GORZ}.  The semicircle
law for Hermite zeros then gives a global GUE statement after taking
$d\to\infty$ in a second limit.  The uniform comparison used here
allows the two parameters to tend to infinity together.

\begin{theorem}[Joint semicircle limit]
\label{thm:semicircle}
Let $K$ be as in Theorem~\ref{thm:main}, and let
$(n_\nu,d_\nu)$ be any sequence of positive integers such that
\[
 n_\nu\longrightarrow\infty,\qquad
 d_\nu\longrightarrow\infty,\qquad
 n_\nu^3\log^2(n_\nu+2)\geq Kd_\nu^5.
\]
Write $\rho_{1,\nu},\ldots,\rho_{d_\nu,\nu}$ for the zeros of
$J^{d_\nu,n_\nu}$ and put
\[
 R_{1,\nu}:=\frac{\gamma(n_\nu+1)}{\gamma(n_\nu)},\qquad
 \lambda_{j,\nu}:=
 \sqrt{\frac{n_\nu}{d_\nu}}
 \bigl(1+R_{1,\nu}\rho_{j,\nu}\bigr).
\]
Then, for every bounded continuous function $f\colon\R\to\R$,
\[
 \lim_{\nu\to\infty}\frac1{d_\nu}
 \sum_{j=1}^{d_\nu}f(\lambda_{j,\nu})
 =\frac1{2\pi}\int_{-2}^{2}f(x)\sqrt{4-x^2}\,\dd x.
\]
\end{theorem}

Theorem~\ref{thm:semicircle} concerns the global empirical measure.
It does not assert local sine-kernel statistics, eigenvalue-spacing
laws, or edge universality.

The proof separates the construction of a real-rooted comparison
polynomial from the stability of its zeros under the remaining
coefficient perturbation.  Section~\ref{sec:architecture} introduces
the quotient coordinates that organize the construction, displays the
three nested comparison families, and proves the abstract stability
principle used at the end.  The subsequent sections verify its
hypotheses.  The analytic input is the complex form of the saddle
estimate in \cite[Section~3]{GORTTW}; the algebraic input consists of
Jacobi root estimates and the real-root and logarithmic-mesh
properties of finite-free multiplicative convolution
\cite[Propositions~2.7 and 2.17]{MMP}.
After the hyperbolicity proof is complete,
Section~\ref{sec:joint-semicircle} identifies the semicircle law for
the comparison model and transfers it to the Jensen zeros.

\section{Comparison coordinates and proof architecture}
\label{sec:architecture}

Fix $n$ and normalize the coefficient sequence by
\begin{equation}\label{eq:Rq-global}
 R_0:=1,\qquad
 R_j:=\frac{\gamma(n+j)}{\gamma(n)},\qquad
 q_k:=\frac{R_{k+1}^2}{R_kR_{k+2}}
 \quad(j\geq1,\ k\geq0).
\end{equation}
Then
\[
 \frac{J^{d,n}(X)}{\gamma(n)}
 =\sum_{j=0}^d\binom djR_jX^j.
\]
Multiplying $R_j$ by $S^j$ only rescales the polynomial variable and
does not change any $q_k$.  These quotient invariants are the natural
coordinates for the comparison.

We note the following, whose proof is immediate:
\begin{lemma}[Quotient coordinates]\label{lem:quotient-coordinates}
Let $\widetilde R_0=R_0=1$ and $\widetilde R_1=R_1$.  If
\[
 \frac{\widetilde R_{k+1}^2}
 {\widetilde R_k\widetilde R_{k+2}}=q_k
 \qquad(0\leq k\leq m-2),
\]
then $\widetilde R_j=R_j$ for $0\leq j\leq m$.
\end{lemma}

% \begin{proof}
% Starting from $R_0,R_1$, recover the sequence successively from
% \[
%  R_{k+2}=\frac{R_{k+1}^2}{q_kR_k}.
% \]
% \end{proof}

The construction passes through three nested ratio families:
\begin{align}
 \mathsf L_j(B,S)
 &:=\frac{S^j}{\rf{B}{j}},
 &&B,S>0,                                      \label{eq:Laguerre-model}\\
 \mathsf J_j(A,B,S)
 &:=\frac{S^j\rf{A}{j}}{A^j\rf{B}{j}},
 &&A,B,S>0,                                    \label{eq:Jacobi-model-family}\\
 \mathsf F_j(A,B,C,D,S)
 &:=S^j\frac{\rf{A}{j}\rf{C}{j}D^j}
 {A^jC^j\rf{B}{j}\rf{D}{j}},
 &&A,B,C,D,S>0.                                \label{eq:free-model-family}
\end{align}
The Laguerre family is the limit of the Jacobi family as
$A\to\infty$, while
\[
 \mathsf F_j(A,B,C,C,S)=\mathsf J_j(A,B,S).
\]
The polynomial associated with the first family is
\begin{align}
 \sum_{j=0}^d\binom dj\mathsf L_j(B,S)X^j
 &= {}_1F_1(-d;B;-SX)\notag\\
 &=\frac{d!}{\rf{B}{d}}L_d^{\,B-1}(-SX),
\label{eq:Laguerre-model-polynomial}
\end{align}
so it has $d$ simple negative zeros.

The role of the three families can now be stated without their
parameter calculations.  The Laguerre family is scaled to match
$R_1$ and $R_2$.  Its error at $R_3$ has a definite positive sign;
that sign permits a Jacobi parameter $A>0$ to match
$R_1,R_2,R_3$.  The Jacobi model has a defect of the opposite sign
at $R_4$.  A second Jacobi factor, introduced through finite-free
multiplicative convolution, absorbs that defect and produces a model
matching $R_0,\ldots,R_4$ exactly.  Lemma~\ref{lem:quotient-coordinates}
reduces each matching assertion to the corresponding equations for
the $q_k$.

It remains to explain why matching five coefficients is enough.
The following proposition isolates the final argument from the
special functions used to construct the model.

\begin{proposition}[Fifth-order multiplier stability]
\label{prop:abstract-stability}
Let $d\geq5$, and let
\[
 p(y)=\sum_{j=0}^d p_jy^j
\]
be a real polynomial with $d$ simple positive zeros and $p(0)\neq0$.
For $r>0$, put
\[
 \Omega_r:=\{z\in\C:\dist(z,[0,d])\leq2r\}.
\]
Suppose that $c$ is holomorphic on a neighborhood of $\Omega_r$,
that $c(0),\ldots,c(d)$ are real, and that
\[
 c(0)=\cdots=c(4)=1,\qquad c(d)>0,\qquad
 \sup_{\Omega_r}|c-1|\leq\varepsilon<16.
\]
Define
\[
 P(y):=\sum_{j=0}^dp_jc(j)y^j.
\]
If, at every critical point $y$ of $p$,
\begin{equation}\label{eq:abstract-derivative-ratios}
 \left|\frac{y^kp^{(k)}(y)}{p(y)}\right|\leq r^k
 \qquad(0\leq k\leq d),
\end{equation}
then $P$ has $d$ simple positive zeros.  Moreover, if $N_Q(t)$
denotes the number of zeros of a polynomial $Q$ in $(0,t]$, counted
with multiplicity, then
\begin{equation}\label{eq:abstract-counting-transfer}
 \sup_{t>0}|N_P(t)-N_p(t)|\leq1.
\end{equation}
\end{proposition}

\begin{proof}
Write
\[
 \Delta^kc(0)
 :=\sum_{\ell=0}^k(-1)^{k-\ell}\binom k\ell c(\ell).
\]
Newton interpolation gives
\begin{equation}\label{eq:abstract-newton}
 P(y)=\sum_{k=0}^d
 \Delta^kc(0)\frac{y^kp^{(k)}(y)}{k!}.
\end{equation}
Indeed, substitute
$c(j)=\sum_{k=0}^j\binom jk\Delta^kc(0)$ and interchange the two
finite sums.

For $k\geq1$, repeated use of the fundamental theorem of calculus
gives
\[
 \Delta^kc(0)
 =\int_{[0,1]^k}c^{(k)}(t_1+\cdots+t_k)
   \dd t_1\cdots\dd t_k.
\]
Cauchy's estimate on disks of radius $\rho<2r$ centered on
$[0,d]$, followed by $\rho\uparrow2r$, therefore yields
\begin{equation}\label{eq:abstract-difference-bound}
 \frac{|\Delta^kc(0)|}{k!}
 \leq\frac{\varepsilon}{(2r)^k}.
\end{equation}
The five matching conditions imply
\[
 \Delta^0c(0)=1,\qquad
 \Delta^kc(0)=0\quad(1\leq k\leq4).
\]
At a critical point of $p$, equations
\eqref{eq:abstract-newton}--\eqref{eq:abstract-difference-bound}
and \eqref{eq:abstract-derivative-ratios} give
\[
 \left|\frac{P(y)}{p(y)}-1\right|
 \leq\varepsilon\sum_{k=5}^d2^{-k}
 <\frac{\varepsilon}{16}<1.
\]
Thus $P$ and $p$ have the same sign at every critical point.
They also have the same sign at $0$ and for large positive $y$,
because $c(0)=1$ and $c(d)>0$.  The signs of $p$ at $0$, at its
$d-1$ critical points, and at $+\infty$ alternate.  Hence the
intermediate value theorem gives a sign-changing zero of $P$ in each
of the $d$ intervening intervals.  These exhaust its degree, so all
are simple.  The zeros of $p'$ strictly interlace those of $p$, so
$p$ also has one zero in each of these intervals.  If $t$ lies in
one of them, $p$ and $P$ have the same number of zeros in all
preceding intervals, and their counts inside the interval containing
$t$ differ by at most one.  This proves
\eqref{eq:abstract-counting-transfer}.
\end{proof}

The rest of the proof verifies the hypotheses of
Proposition~\ref{prop:abstract-stability}.  The saddle estimate gives
the two defect signs needed to choose the model parameters and,
after exact matching through $R_4$, the bound
\[
 \sup_{\Omega_r}|c-1|
 \ll\frac{d^{5/2}}{n^{3/2}\log(n+2)}.
\]
The finite-free factorization gives the required positive roots, and
a perturbed Jacobi equation gives
\eqref{eq:abstract-derivative-ratios}.  The condition that the
displayed error be uniformly small is exactly
$d^5\ll n^3\log^2(n+2)$.

\section{Moment representation and the Laguerre baseline}

The integral representation below explains the distinguished
Laguerre normalization used to identify the positive defect at
$R_3$.

Write
\[
 \Xi(t):=\xi\!\left(\frac12+it\right).
\]
Riemann's standard integral representation has the form
\[
 \Xi(t)=\int_0^\infty \Phi(u)\cos(tu)\dd u,
 \qquad \Phi(u)>0,
\]
where $\Phi$ is the kernel obtained from the Jacobi theta series;
the integral and all the differentiations used below converge
absolutely.
Put
\[
 M_n:=\int_0^\infty \Phi(u)u^{2n}\dd u.
\]
Since $\xi_{\mathrm c}(z)=\Xi(-iz)$, comparison of the expansions of
$\cosh(zu)$ and $\xi_{\mathrm c}(z)$ gives
\begin{equation}\label{eq:gamma-moment}
 \gamma(n)=\frac{n!}{(2n)!}M_n.
\end{equation}

Define the tilted kernel probability measure
\[
 \dd\mathbb P_n(u)
 :=\frac{\Phi(u)u^{2n}}{M_n}\dd u
\]
and let $U$ denote its coordinate random variable.  Legendre's
duplication formula
\[
 \Gamma(z)\Gamma\!\left(z+\frac12\right)
 =2^{1-2z}\sqrt\pi\,\Gamma(2z)
\]
gives
\begin{equation}\label{eq:gamma-ratio-mixture}
 \frac{\gamma(n+j)}{\gamma(n)}
 =\frac{1}{4^j\rf{n+\frac12}{j}}\,
   \E_n[U^{2j}].
\end{equation}
Consequently,
\begin{align}
 \frac{J^{d,n}(X)}{\gamma(n)}
 &=\E_n\!\left[
   \sum_{j=0}^{d}\binom dj
   \frac{(XU^2)^j}{4^j\rf{n+\frac12}{j}}
   \right] \notag\\
 &=\E_n\!\left[
 {}_1F_1\!\left(-d;n+\frac12;-\frac{XU^2}{4}\right)
 \right] \notag\\
 &=\frac{d!}{\rf{n+\frac12}{d}}\,
   \E_n\!\left[
   L_d^{\,n-\frac12}\!\left(-\frac{XU^2}{4}\right)
   \right].
\label{eq:kernel-laguerre-mixture}
\end{align}
For every fixed $U>0$, the polynomial inside the expectation in
\eqref{eq:kernel-laguerre-mixture} has $d$ distinct negative roots,
although averaging does not in general preserve hyperbolicity.

There is, however, an exact hyperbolicity-preserving random-dilation
model.
Let
\[
 b:=n+\frac12,\qquad B>b,
\]
and let $Z$ have the beta distribution $\operatorname{Beta}(b,B-b)$.
Then
\[
 \E[Z^j]=\frac{\rf b j}{\rf B j}.
\]
It follows immediately that, for $a>0$,
\begin{equation}\label{eq:sonine}
 \E\!\left[
 {}_1F_1\!\left(-d;b;-\frac{aZX}{4}\right)
 \right]
 =
 {}_1F_1\!\left(-d;B;-\frac{aX}{4}\right).
\end{equation}
This is the beta--Sonine index-raising operation for Laguerre
polynomials.  With $S=a/4$, it realizes the Laguerre family
\eqref{eq:Laguerre-model} as a scaled beta dilation of the baseline
in \eqref{eq:kernel-laguerre-mixture}.  This explains why that family
is the first comparison model; no factorization of the tilted kernel
measure is asserted.

\section{The analytic saddle and the first defect}

The saddle calculation is used in three places: to determine the
sign of the Laguerre error at $R_3$, to determine the sign of the
Jacobi error at $R_4$, and to bound the fifth derivative of the final
residual multiplier.  We record all the required derivative
information here.

Continue the moment in \eqref{eq:gamma-moment} by its absolutely
convergent Mellin integral and define
\[
 \gamma(z):=\frac{\Gamma(z+1)}{\Gamma(2z+1)}M_z.
\]
The same duplication formula gives
\begin{equation}\label{eq:gamma-h}
 \gamma(z)
 =\frac{\sqrt\pi\,M_z}{4^z\Gamma(z+\frac12)},
 \qquad
 M_z:=\int_0^\infty\Phi(u)u^{2z}\dd u.
\end{equation}
To identify this continuation with the saddle formula in
\cite[Section~3]{GORTTW}, set
\[
 F(s):=\int_1^\infty(\log t)^s t^{-3/4}
       \sum_{k\geq1}e^{-\pi k^2t}\dd t.
\]
The continuation of the exact kernel identity in
\cite[(3.1)]{GORTTW}, with the normalization corrected as explained
below, is
\begin{equation}\label{eq:corrected-M-via-F}
 M_z=2^{-2z-2}\left\{
 32\binom{2z}{2}F(2z-2)-F(2z)
 \right\}.
\end{equation}
For real $x>3$, let $L_x$ be the positive solution of
\begin{equation}\label{eq:L-def}
 x=L_x\left(\pi e^{L_x}+\frac34\right).
\end{equation}
For complex $x$ in the right half-plane, $L_x$ denotes the branch of
the holomorphic continuation used in \cite[Section~3]{GORTTW}.  Put
\[
 Q_x:=(1+L_x)x-\frac34L_x^2.
\]
The main saddle term is
\begin{equation}\label{eq:G0-def}
 G_0(x):=(x+1)\log L_x+\frac{L_x}{4}-\frac{x}{L_x}
          -\frac12\log Q_x.
\end{equation}

\begin{proposition}[Sectorial logarithmic saddle]
\label{prop:sectorial-saddle}
There are $0<\theta<\pi/2$ and $x_0>0$ such that $M_z$ is
nonzero in
\[
 \mathfrak S_\theta
 :=\{z\in\C:|z|\geq x_0,\ |\arg z|\leq\theta\}.
\]
On this sector, let $h(z)=\log M_z$ be the branch real on the
positive axis.  If $N=2z-2$, then
\begin{align}
 h(z)
 &=G_0(N)+\log\left(32\binom{2z}{2}\right)
    -(2z+2)\log2+c_{\mathrm{sad}}+\mathcal R(z),
\label{eq:sectorial-h}
\end{align}
where $c_{\mathrm{sad}}$ is real.  For every fixed
$0<\varepsilon_0<1/2$ and $0\leq r\leq5$,
\begin{equation}\label{eq:sectorial-remainder}
 \mathcal R^{(r)}(z)
 =O_{\varepsilon_0,r,\theta}
   \left(|z|^{-r-1+\varepsilon_0}\right)
\end{equation}
uniformly on a slightly smaller closed sector.
\end{proposition}

Here $h$ is the exact logarithm of the continued moment, whereas
$G_0$ is only its principal saddle contribution, evaluated at
$N=2z-2$.  The remaining explicit terms account for normalization,
and $\mathcal R$ is the controlled analytic remainder.

\begin{proof}
Write $N=2z-2$ and
\[
 K_N:=\left(L_N^{-1}+L_N^{-2}\right)N-\frac34
     =\frac{Q_N}{L_N^2}.
\]
The complex form of \cite[(3.2)]{GORTTW}, stated in the sentence
following that equation, is uniform on every fixed closed sector in
$\Re z>1$ and reads, in the normalization of this paper,
\begin{align}
 \gamma(z)
 &=\frac{e^{z-2}z^{z+1/2}L_N^N}
 {2^{2z-2}N^{N+1/2}}
 \left(\frac{2\pi}{K_N}\right)^{1/2}
 \exp\left(\frac{L_N}{4}-\frac{N}{L_N}+\frac34\right)
 \notag\\
 &\qquad\times
 \left(1+O_{\varepsilon_0}
       (|z|^{-1+\varepsilon_0})\right).
\label{eq:complex-gamma-saddle}
\end{align}
Here all powers use the branches obtained by continuation from the
positive axis.  The factor $2^{2z-2}$ incorporates a normalization
that is easy to miss: with
$\Lambda(s)=\pi^{-s/2}\Gamma(s/2)\zeta(s)$, the coefficients in
\cite{GORZ} are those of
$(4w^2-1)\Lambda(\frac12+w)=8\xi(\frac12+w)$.  Thus the right sides
of \cite[(3.1)--(3.2)]{GORTTW}, when used for the coefficients
defined here, are divided by $8$.  This affects only an additive
constant after taking logarithms.  It does not affect the results
of \cite{GORTTW}, whose use of these coefficients in Section~3 is
through ratios in which the constant cancels.

The leading factor in \eqref{eq:complex-gamma-saddle} is holomorphic
and nonzero on a sufficiently small fixed sector.  For large $|z|$
the relative error has modulus below $1/2$, so $\gamma(z)$ and,
by \eqref{eq:gamma-h}, $M_z$ are nonzero there.  Taking the
logarithm, using $K_N=Q_N/L_N^2$, and applying sectorial Stirling
asymptotics to $\Gamma(z+\frac12)$ in \eqref{eq:gamma-h} gives
\eqref{eq:sectorial-h}.  Consolidating the remaining elementary
terms into the explicit terms following $G_0(N)$ changes only
$c_{\mathrm{sad}}+O(|z|^{-1})$.

The difference between the two sides of
\eqref{eq:complex-gamma-saddle}, divided by its nonzero leading
factor, is holomorphic.  Hence so is the logarithmic remainder
$\mathcal R$.  A disk $|\zeta-z|\leq\delta|z|$ about a point in a
smaller closed sector stays in the original sector, and
$|\zeta|\asymp|z|$ there.  Cauchy's inequalities applied to the
$O(|z|^{-1+\varepsilon_0})$ remainder prove
\eqref{eq:sectorial-remainder}.
\end{proof}

\begin{lemma}[Signed moment-saddle derivatives]
\label{lem:signed-saddle}
Put
\[
 \mathcal L_x:=L_{2x-2}.
\]
As $x\to+\infty$,
\begin{align}
 h''(x)
 &=\frac{2}{x\mathcal L_x}
   \left(1+O\!\left(\frac1{\mathcal L_x}\right)\right),
\label{eq:hsecond-asymptotic}\\
 h'''(x)
 &=-\frac{2}{x^2\mathcal L_x}
   \left(1+O\!\left(\frac1{\mathcal L_x}\right)\right).
\label{eq:hthird-asymptotic}
\end{align}
In addition,
\begin{equation}\label{eq:hfourth-asymptotic}
 h^{(4)}(x)
 =\frac{4}{x^3\mathcal L_x}
   \left(1+O\!\left(\frac1{\mathcal L_x}\right)\right).
\end{equation}
There are absolute constants $\eta,C>0$ such that
\begin{align}
 |h^{(4)}(x+w)|
 &\leq\frac{C}{x^3\log x},\label{eq:hfourth}\\
 |h^{(5)}(x+w)|
 &\leq\frac{C}{x^4\log x}
\label{eq:hfifth}
\end{align}
whenever $|w|\leq\eta x$.
\end{lemma}

\begin{proof}
Differentiating the saddle equation gives
\[
 L_N'=\frac{L_N}{Q_N}.
\]
Using this identity in $G_0$ gives
\begin{align}
 G_0''(N)
 &=\frac1{Q_N}+O(N^{-2}),\label{eq:Gsecond}\\
 G_0'''(N)
 &=-\frac1{N^2L_N}
   \left(1+O(L_N^{-1})\right),\label{eq:Gthird}\\
 G_0^{(4)}(N)
 &=\frac2{N^3L_N}
   \left(1+O(L_N^{-1})\right),\label{eq:Gfourth}\\
 G_0^{(5)}(N)
 &=O\!\left(\frac1{N^4L_N}\right).
\label{eq:Gfifth}
\end{align}
Indeed, $Q_N=NL_N(1+O(L_N^{-1}))$ and
$Q_N'=L_N(1+O(L_N^{-1}))$; the last two formulas follow by further
differentiation of \eqref{eq:Gthird}.

The asserted estimates follow from
Proposition~\ref{prop:sectorial-saddle} and
\eqref{eq:Gsecond}--\eqref{eq:Gfifth}; the chain rule contributes
the factors $2^m$.  The elementary terms following $G_0(2z-2)$ in
\eqref{eq:sectorial-h}, and the remainder controlled by
\eqref{eq:sectorial-remainder}, are of lower order.  A sufficiently
small fixed $\eta>0$ makes every disk $|w|\leq\eta x$ lie in the
sector, which proves the two uniform bounds.
\end{proof}

\begin{lemma}[Positive Laguerre defect at $R_3$]
\label{lem:third-defect}
For all sufficiently large $n$, $q_0>1$.  Set
\[
 B_{\mathrm L}:=(q_0-1)^{-1},\qquad
 S_{\mathrm L}:=B_{\mathrm L}R_1,
 \qquad
 R_j^{(\mathrm L)}:=\mathsf L_j(B_{\mathrm L},S_{\mathrm L}).
\]
Then the Laguerre model matches the first two normalized
coefficients,
\[
 R_1^{(\mathrm L)}=R_1,\qquad R_2^{(\mathrm L)}=R_2,
\]
and predicts
\[
 q_1^{(\mathrm L)}
 :=\frac{(R_2^{(\mathrm L)})^2}
          {R_1^{(\mathrm L)}R_3^{(\mathrm L)}}
 =\frac{B_{\mathrm L}+2}{B_{\mathrm L}+1}
 =2-\frac1{q_0}.
\]
Moreover,
\begin{equation}\label{eq:Dn-asymptotic}
 \mathcal D_n
 :=\log\frac{R_3}{R_3^{(\mathrm L)}}
 =\log\frac{q_1^{(\mathrm L)}}{q_1}
 =
 \frac{2}{n^2\mathcal L_n}
 \left(1+O\!\left(\frac1{\mathcal L_n}\right)\right).
\end{equation}
In particular, $\mathcal D_n>0$ for all sufficiently large $n$.
\end{lemma}

\begin{proof}
Put $b=n+\frac12$ and
\[
 \ell(x):=\log(1+x^{-1}),\qquad
 \delta_0:=h(n+2)-2h(n+1)+h(n),
\]
\[
 \delta_1:=h(n+3)-2h(n+2)+h(n+1).
\]
Equation \eqref{eq:gamma-h} gives
\begin{equation}\label{eq:q01-delta}
 \log q_0=\ell(b)-\delta_0,\qquad
 \log q_1=\ell(b+1)-\delta_1.
\end{equation}
Lemma~\ref{lem:signed-saddle} and the integral formulas for finite
differences give
\begin{align}
 \delta_0
 &=\frac{2}{n\mathcal L_n}
   \left(1+O(\mathcal L_n^{-1})\right),\label{eq:delta0}\\
 \delta_1-\delta_0
 &=-\frac{2}{n^2\mathcal L_n}
   \left(1+O(\mathcal L_n^{-1})\right).
\label{eq:delta10}
\end{align}
Since $\ell(b)\sim n^{-1}$, equations
\eqref{eq:q01-delta} and \eqref{eq:delta0} show that $q_0>1$ for
all sufficiently large $n$.  Direct substitution in
\eqref{eq:Laguerre-model} gives
$R_1^{(\mathrm L)}=R_1$, $R_2^{(\mathrm L)}=R_2$, and the displayed
formula for $q_1^{(\mathrm L)}$.

Define $\mathcal F(x):=\log(2-e^{-x})$.  The exact relation
$q_1^{(\mathrm L)}=2-q_0^{-1}$ and
$\mathcal F(\ell(b))=\ell(b+1)$ imply
\begin{align*}
 \mathcal D_n
 &=
 \mathcal F(\ell(b)-\delta_0)-\ell(b+1)+\delta_1\\
 &=
 \delta_1-\frac{b}{b+2}\delta_0+O(\delta_0^2)\\
 &=
 (\delta_1-\delta_0)+\frac{2}{b+2}\delta_0
 +O(\delta_0^2).
\end{align*}
Substitution of \eqref{eq:delta0} and \eqref{eq:delta10} gives
\eqref{eq:Dn-asymptotic}.
\end{proof}

\section{First correction: the three-coefficient Jacobi model}

We now absorb the positive $R_3$ defect in
Lemma~\ref{lem:third-defect}.  For the Jacobi family
\eqref{eq:Jacobi-model-family}, the first two quotient invariants are
\begin{align}
 q_0^{(\mathrm J)}(A,B)
 &:=
 \frac{A(B+1)}{B(A+1)},\label{eq:q0-model}\\
 q_1^{(\mathrm J)}(A,B)
 &:=
 \frac{(A+1)(B+2)}{(B+1)(A+2)}.
\label{eq:q1-model}
\end{align}
They satisfy
\begin{equation}\label{eq:Jacobi-q-relation}
 \frac{q_0^{(\mathrm J)}-1}{q_1^{(\mathrm J)}-1}
 =q_0^{(\mathrm J)}\left(1+\frac2A\right).
\end{equation}
Consequently, solving
$q_0^{(\mathrm J)}=q_0$ and $q_1^{(\mathrm J)}=q_1$ gives the
following formulas.  Put
\begin{equation}\label{eq:T-def}
 T:=\frac{q_0-1}{q_1-1}.
\end{equation}
For all sufficiently large $n$, define
\begin{equation}\label{eq:Jacobi-parameters}
 A_{\mathrm J}:=\frac{2q_0}{T-q_0},\qquad
 \frac1{B_{\mathrm J}}
 :=q_0\left(1+\frac1{A_{\mathrm J}}\right)-1,
 \qquad
 S_{\mathrm J}:=B_{\mathrm J}R_1.
\end{equation}

\begin{lemma}[Positive Jacobi parameters]\label{lem:Jacobi-parameters}
The parameters in \eqref{eq:Jacobi-parameters} are positive and
satisfy
\begin{equation}\label{eq:AB-asymptotic}
 A_{\mathrm J}\sim n\mathcal L_n,\qquad
 B_{\mathrm J}\sim n,\qquad
 B_{\mathrm J}-\left(n+\frac12\right)
 =O\!\left(\frac n{\log n}\right).
\end{equation}
\end{lemma}

\begin{proof}
Let $q_1^{(\mathrm L)}=2-q_0^{-1}$.  Since
\eqref{eq:q01-delta}, \eqref{eq:delta0}, and \eqref{eq:delta10} give
$\log q_1\sim n^{-1}$, we have $q_1>1$.  On the other hand,
$\mathcal D_n=\log(q_1^{(\mathrm L)}/q_1)>0$ gives
$q_1<q_1^{(\mathrm L)}$.  The identity
\[
 q_1^{(\mathrm L)}-1=\frac{q_0-1}{q_0}
\]
then gives $T>q_0$.  Thus $A_{\mathrm J}>0$, and the remaining two
parameters in \eqref{eq:Jacobi-parameters} are positive.

For later use, observe that
\begin{equation}\label{eq:T-minus-q}
 T-q_0
 =
 \frac{(q_0-1)(q_1^{(\mathrm L)}-q_1)}
 {(q_1-1)(q_1^{(\mathrm L)}-1)}.
\end{equation}
Equations \eqref{eq:q01-delta} and \eqref{eq:delta0} imply
$q_0-1\sim n^{-1}$.  Lemma~\ref{lem:third-defect} gives
\[
 q_1^{(\mathrm L)}-q_1
 =q_1^{(\mathrm L)}(1-e^{-\mathcal D_n})
 \sim\frac{2}{n^2\mathcal L_n}.
\]
Both $q_1-1$ and $q_1^{(\mathrm L)}-1$ are asymptotic to $n^{-1}$.
Hence \eqref{eq:T-minus-q} gives
\[
 T-q_0\sim\frac{2}{n\mathcal L_n},
\]
and therefore $A_{\mathrm J}\sim n\mathcal L_n$.

The second equation in \eqref{eq:Jacobi-parameters}, together with
$q_0-1\sim n^{-1}$ and $A_{\mathrm J}\asymp n\log n$, gives
$B_{\mathrm J}\sim n$.  More precisely,
\eqref{eq:q01-delta} shows
\[
 q_0-1=\frac1{n+\frac12}
 +O\!\left(\frac1{n\log n}\right).
\]
It follows that
\[
 \frac1{B_{\mathrm J}}
 =\frac1{n+\frac12}
 +O\!\left(\frac1{n\log n}\right),
\]
which proves the last estimate in \eqref{eq:AB-asymptotic}.
\end{proof}

\begin{lemma}[Refined Jacobi scales]\label{lem:Jacobi-refined}
As $n\to\infty$,
\begin{align}
 A_{\mathrm J}
 &=n\mathcal L_n
   \left(1+O(\mathcal L_n^{-1})\right),\label{eq:AJ-refined}\\
 B_{\mathrm J}
 &=n+\frac{n}{\mathcal L_n}
   \left(1+O(\mathcal L_n^{-1})\right).
\label{eq:BJ-refined}
\end{align}
Consequently, with $b=n+\frac12$,
\[
 B_{\mathrm J}-b
 =\frac{n}{\mathcal L_n}
  \left(1+O(\mathcal L_n^{-1})\right).
\]
\end{lemma}

\begin{proof}
Keep the notation in the proof of
Lemma~\ref{lem:Jacobi-parameters}.  Equations
\eqref{eq:delta0}, \eqref{eq:delta10}, and
\eqref{eq:Dn-asymptotic}, with one further term absorbed in their
relative $O(\mathcal L_n^{-1})$ errors, give
\[
 T-q_0
 =\frac{2}{n\mathcal L_n}
  \left(1+O(\mathcal L_n^{-1})\right).
\]
The first formula in \eqref{eq:Jacobi-parameters} proves
\eqref{eq:AJ-refined}.

Also,
\[
 q_0
 =\left(1+\frac1b\right)e^{-\delta_0},
\]
and hence
\[
 q_0-1
 =\frac1n-\frac{2}{n\mathcal L_n}
  +O\!\left(\frac1{n\mathcal L_n^2}+\frac1{n^2}\right).
\]
Using \eqref{eq:AJ-refined} in the second formula of
\eqref{eq:Jacobi-parameters} yields
\[
 \frac1{B_{\mathrm J}}
 =\frac1n-\frac1{n\mathcal L_n}
  +O\!\left(\frac1{n\mathcal L_n^2}+\frac1{n^2}\right).
\]
Inversion proves \eqref{eq:BJ-refined}.
\end{proof}

Define the Jacobi model ratios
\begin{equation}\label{eq:Jacobi-model-ratios}
 R_j^{(\mathrm J)}
 :=\mathsf J_j(A_{\mathrm J},B_{\mathrm J},S_{\mathrm J}).
\end{equation}

\begin{lemma}[Exact matching through $R_3$]\label{lem:match-three}
For $j=0,1,2,3$,
\[
 R_j^{(\mathrm J)}=R_j.
\]
\end{lemma}

\begin{proof}
The choice $S_{\mathrm J}=B_{\mathrm J}R_1$ matches $R_1$.
Equations \eqref{eq:Jacobi-parameters} and
\eqref{eq:Jacobi-q-relation} match $q_0$ and $q_1$.
Lemma~\ref{lem:quotient-coordinates} now gives the assertion.
\end{proof}

To determine whether a second positive Jacobi factor can correct
$R_4$, continue the logarithm of the ratio between the actual
coefficients and the first Jacobi model.  For $z$ in a fixed
neighborhood of $\{0,1,2,3,4\}$, define
\begin{align}
 E_{\mathrm J}(z)
 &:=
 \log\gamma(n+z)-\log\gamma(n)-z\log S_{\mathrm J}
 +z\log A_{\mathrm J}\notag\\
 &\quad
 +\log\Gamma(B_{\mathrm J}+z)-\log\Gamma(B_{\mathrm J})
 -\log\Gamma(A_{\mathrm J}+z)+\log\Gamma(A_{\mathrm J}),
\label{eq:EJ-def}
\end{align}
where the branches are real on the positive axis.  We write
$\psi=\Gamma'/\Gamma$ and $\psi^{(m)}$ for its $m$th derivative.
For integer $j\geq0$,
\[
 E_{\mathrm J}(j)
 =\log\frac{R_j}{R_j^{(\mathrm J)}}.
\]
Lemma~\ref{lem:match-three} therefore gives
\begin{equation}\label{eq:EJ-zeros}
 E_{\mathrm J}(0)=E_{\mathrm J}(1)
 =E_{\mathrm J}(2)=E_{\mathrm J}(3)=0.
\end{equation}
Writing $b=n+\tfrac12$, four differentiations of
\eqref{eq:EJ-def} and \eqref{eq:gamma-h} give
\begin{align}
 E_{\mathrm J}^{(4)}(z)
 &=
 h^{(4)}(n+z)
 +\psi^{(3)}(B_{\mathrm J}+z)-\psi^{(3)}(b+z)\notag\\
 &\hspace{28mm}
 -\psi^{(3)}(A_{\mathrm J}+z).
\label{eq:EJ-fourth}
\end{align}

\begin{lemma}[Negative Jacobi defect at $R_4$]
\label{lem:signed-fourth-defect}
For fixed $z$ in a bounded set,
\begin{equation}\label{eq:EJ-fourth-signed}
 E_{\mathrm J}^{(4)}(z)
 =-\frac{2}{n^3\mathcal L_n}
  \left(1+O(\mathcal L_n^{-1})\right).
\end{equation}
In particular,
\begin{equation}\label{eq:fourth-defect}
 \mathcal Q_n:=-E_{\mathrm J}(4)
 =\frac{2}{n^3\mathcal L_n}
  \left(1+O(\mathcal L_n^{-1})\right)>0
\end{equation}
for all sufficiently large $n$.
\end{lemma}

\begin{proof}
Equation \eqref{eq:EJ-fourth}, together with
\eqref{eq:hfourth-asymptotic}, gives
\[
 h^{(4)}(n+z)
 =\frac4{n^3\mathcal L_n}
  \left(1+O(\mathcal L_n^{-1})\right)
\]
uniformly for bounded $z$.  By Lemma~\ref{lem:Jacobi-refined} and
the mean-value formula,
\begin{align*}
 \psi^{(3)}(B_{\mathrm J}+z)-\psi^{(3)}(b+z)
 &=(B_{\mathrm J}-b)\psi^{(4)}(b+z)
   +O\!\left((B_{\mathrm J}-b)^2
       \sup_{u\asymp n}|\psi^{(5)}(u)|\right)\\
 &=-\frac6{n^3\mathcal L_n}
   \left(1+O(\mathcal L_n^{-1})\right).
\end{align*}
Here we used
\[
 \psi^{(4)}(u)=-\frac6{u^4}+O(u^{-5}),
 \qquad
 \psi^{(5)}(u)=O(u^{-5}).
\]
Finally,
$\psi^{(3)}(A_{\mathrm J}+z)=O(n^{-3}\mathcal L_n^{-3})$.
This proves \eqref{eq:EJ-fourth-signed}.

Since $E_{\mathrm J}$ vanishes at $0,1,2,3$ by
\eqref{eq:EJ-zeros}, the
Hermite--Genocchi formula writes $E_{\mathrm J}(4)$ as the average
of $E_{\mathrm J}^{(4)}$ over the simplex spanned by
$0,1,2,3,4$.  The same uniform asymptotic therefore proves
\eqref{eq:fourth-defect}.
\end{proof}

\section{Second correction: the four-coefficient finite-free model}

We now use the sign in Lemma~\ref{lem:signed-fourth-defect} to
match $R_4$.  For $U,V>0$, define
\begin{equation}\label{eq:qUV}
 \mathfrak q_k(U,V)
 :=
 \frac{(U+k)(V+k+1)}
 {(V+k)(U+k+1)}.
\end{equation}
This is the $k$th quotient invariant of the Jacobi family
$\mathsf J_j(U,V,S)$.  Since
\[
 \mathsf F_j(A,B,C,D,S)
 =\mathsf J_j(A,B,S)\mathsf J_j(C,D,D),
\]
the $k$th quotient invariant of the finite-free family is
\[
 \mathfrak q_k(A,B)\mathfrak q_k(C,D).
\]
For the first Jacobi model,
\[
 q_k=\mathfrak q_k(A_{\mathrm J},B_{\mathrm J})
 \qquad(k=0,1).
\]
Moreover, because the Jacobi model agrees with $R_j$ for
$0\leq j\leq3$,
\begin{equation}\label{eq:q2-defect}
 \log\frac{q_2}
 {\mathfrak q_2(A_{\mathrm J},B_{\mathrm J})}
 =-E_{\mathrm J}(4)=\mathcal Q_n.
\end{equation}

Fix
\begin{equation}\label{eq:D-def}
 D:=\frac12B_{\mathrm J}.
\end{equation}
This convenient choice places the second Jacobi factor on the same
$n$-scale as the first and leaves three unknown parameters for the
three quotient equations.  By
Lemma~\ref{lem:quotient-coordinates}, matching through $R_4$
now reduces to finding $A,B,C>0$ satisfying
\begin{equation}\label{eq:four-ratio-system}
 \mathfrak q_k(A,B)\mathfrak q_k(C,D)=q_k,
 \qquad k=0,1,2.
\end{equation}

\begin{lemma}[Four-coefficient parameter matching]
\label{lem:four-ratio-parameters}
There is an absolute constant $K_{\mathrm r}\geq32$ such that, for
all sufficiently large $n$, system
\eqref{eq:four-ratio-system} has a real solution with
\begin{align}
 A&=\frac65A_{\mathrm J}
    \left(1+O(\mathcal L_n^{-1})\right),
\label{eq:A-free-asymptotic}\\
 B-B_{\mathrm J}
 &=\frac{n}{3\mathcal L_n}
   \left(1+O(\mathcal L_n^{-1})\right),
\label{eq:B-free-asymptotic}\\
 C-D
 &=\frac{n}{24\mathcal L_n}
   \left(1+O(\mathcal L_n^{-1})\right).
\label{eq:C-free-asymptotic}
\end{align}
If
\begin{equation}\label{eq:fifth-wedge-hypothesis}
 n^3\log^2(n+2)\geq K_4d^5
\end{equation}
with $K_4$ sufficiently large, then
\begin{equation}\label{eq:free-positive-range}
 A\geq8B,\quad B\geq K_{\mathrm r}d,\quad
 D\geq K_{\mathrm r}d,\quad
 4d\leq C-D\leq\frac14D.
\end{equation}
\end{lemma}

\begin{proof}
Put
\[
 f_k(U):=\log\frac{U+k}{U+k+1}.
\]
After taking logarithms, \eqref{eq:four-ratio-system} becomes
\begin{equation}\label{eq:f-system}
 f_k(A)-f_k(B)+f_k(C)-f_k(D)=\log q_k.
\end{equation}
Set $L=\mathcal L_n$ and introduce scaled variables by
\[
 A=aA_{\mathrm J},\qquad
 B=B_{\mathrm J}+\frac{bn}{L},\qquad
 C=D+\frac{cn}{L}.
\]
The expansion
\begin{equation}\label{eq:fk-expansion}
 f_k(U)
 =-\frac1U+\frac{k+\frac12}{U^2}
  -\frac{k^2+k+\frac13}{U^3}
  +O(U^{-4})
\end{equation}
is uniform, together with its first derivatives in the scaled
variables, when $k=0,1,2$ and $(a,b,c)$ ranges over a fixed compact
subset of $(0,\infty)\times\R^2$.

To make the implicit-function step explicit, subtract the Jacobi
model and set
\[
 H_{n,k}(a,b,c)
 :=
 f_k(aA_{\mathrm J})
 -f_k\!\left(B_{\mathrm J}+\frac{bn}{L}\right)
 +f_k\!\left(D+\frac{cn}{L}\right)-f_k(D)
 -f_k(A_{\mathrm J})+f_k(B_{\mathrm J}).
\]
Equations \eqref{eq:f-system} and \eqref{eq:q2-defect} are
equivalent to
\[
 H_{n,0}=H_{n,1}=0,\qquad H_{n,2}=\mathcal Q_n.
\]
The value and the first two finite differences in $k$ occur at
successive powers of $n^{-1}$.  Introduce the correspondingly
rescaled map
\[
 \mathcal G_n(a,b,c)
 :=
 \left(
 nL H_{n,0},\
 n^2L(H_{n,1}-H_{n,0}),\
 n^3L(H_{n,2}-2H_{n,1}+H_{n,0})
 \right).
\]
By Lemma~\ref{lem:signed-fourth-defect}, the required target is
\[
 \left(0,0,n^3L\mathcal Q_n\right)
 =\left(0,0,2+O(L^{-1})\right).
\]
The expansion \eqref{eq:fk-expansion} and
Lemma~\ref{lem:Jacobi-refined} show, uniformly with first
derivatives on a neighborhood of the limiting solution, that
\[
 \mathcal G_n(a,b,c)
 =
 \left(
 1-a^{-1}-b+4c,\
 2b-16c,\
 -6b+96c
 \right)+O(L^{-1}).
\]
Thus the limiting equations are
\begin{align}
 1-a^{-1}-b+4c&=0,\label{eq:limit-system-0}\\
 2b-16c&=0,\label{eq:limit-system-1}\\
 -6b+96c&=2.\label{eq:limit-system-2}
\end{align}
The unique solution is
\[
 x_\ast:= (a_\ast,b_\ast,c_\ast)
 =\left(\frac65,\frac13,\frac1{24}\right).
\]

We now quantify the perturbation argument.  Let $\Phi_n$ be
$\mathcal G_n$ minus its required target and let $\Phi_\infty$ be
the vector of the three limiting left sides in
\eqref{eq:limit-system-0}--\eqref{eq:limit-system-2}.  Thus
\begin{align*}
 \Phi_n(a,b,c)
 &=\mathcal G_n(a,b,c)-\left(0,0,n^3L\mathcal Q_n\right),\\
 \Phi_\infty(a,b,c)
 &=\left(1-a^{-1}-b+4c,\ 2b-16c,\ -6b+96c-2\right).
\end{align*}
Write $\|\cdot\|_\infty$ for the maximum norm and its induced
matrix norm, and set
\[
 \mathcal B:=\left\{x\in\R^3:
 \|x-x_\ast\|_\infty\leq\frac3{20}\right\}.
\]
Define the actual $C^1$ error on this cube by
\begin{equation}\label{eq:matching-C1-error}
 \varepsilon_n
 :=\sup_{x\in\mathcal B}
 \max\left\{
  \|\Phi_n(x)-\Phi_\infty(x)\|_\infty,
  \|D\Phi_n(x)-D\Phi_\infty(x)\|_\infty
 \right\}.
\end{equation}
The uniform expansions above say precisely that
$\varepsilon_n=O(L^{-1})$.  At $x_\ast$ the limiting Jacobian and
its inverse are
\[
 J:=D\Phi_\infty(x_\ast)
 =\begin{pmatrix}
  25/36&-1&4\\
  0&2&-16\\
  0&-6&96
 \end{pmatrix},
 \qquad
 J^{-1}
 =\begin{pmatrix}
  36/25&27/25&3/25\\
  0&1&1/6\\
  0&1/16&1/48
 \end{pmatrix}.
\]
In particular, $\|J^{-1}\|_\infty=66/25$.  Moreover, if
$x=(a,b,c)\in\mathcal B$, then
\begin{equation}\label{eq:limiting-Jacobian-variation}
 \|I-J^{-1}D\Phi_\infty(x)\|_\infty
 =\frac{36}{25}\left|a^{-2}-\frac{25}{36}\right|
 \leq\frac{15}{49}.
\end{equation}

Take $n$ large enough that $\varepsilon_n\leq1/30$.  The map
\[
 \mathcal T_n(x):=x-J^{-1}\Phi_n(x)
\]
satisfies, by \eqref{eq:matching-C1-error} and
\eqref{eq:limiting-Jacobian-variation},
\[
 \sup_{x\in\mathcal B}\|D\mathcal T_n(x)\|_\infty
 \leq\frac{15}{49}+\frac{66}{25}\varepsilon_n
 \leq\frac{15}{49}+\frac{11}{125}<\frac25.
\]
Also
\[
 \|\mathcal T_n(x_\ast)-x_\ast\|_\infty
 \leq\frac{66}{25}\varepsilon_n\leq\frac{11}{125}.
\]
It follows that, for every $x\in\mathcal B$,
\[
 \|\mathcal T_n(x)-x_\ast\|_\infty
 \leq\frac{11}{125}+\frac25\frac3{20}
 =\frac{37}{250}<\frac3{20}.
\]
Thus $\mathcal T_n$ maps $\mathcal B$ into itself and is a
contraction.  Its unique fixed point $x_n$ solves the exact system,
and the contraction estimate gives
\begin{equation}\label{eq:matching-solution-error}
 \|x_n-x_\ast\|_\infty
 \leq
 \frac{(66/25)\varepsilon_n}{1-2/5}
 =\frac{22}{5}\varepsilon_n=O(L^{-1}).
\end{equation}
This proves \eqref{eq:A-free-asymptotic}--
\eqref{eq:C-free-asymptotic} without a qualitative
implicit-function argument.

Finally, \eqref{eq:fifth-wedge-hypothesis} gives
$d=O(n^{3/5}\log^{2/5}n)=o(n/\log n)$.  Since
$A/B\sim(6/5)\mathcal L_n$, we also have $A\geq8B$ beyond a fixed
threshold.
Equations \eqref{eq:A-free-asymptotic}--
\eqref{eq:C-free-asymptotic}, after increasing $K_4$ and excluding
a fixed finite range of $n$, imply
\eqref{eq:free-positive-range}.  The finite range is absorbed by
one further increase of the eventual absolute constant.
\end{proof}

Put
\begin{equation}\label{eq:S-free}
 S:=BR_1
\end{equation}
and define the model ratios
\begin{equation}\label{eq:free-model-ratios}
 R_j^{(\mathrm F)}
 :=\mathsf F_j(A,B,C,D,S).
\end{equation}
The identity $R_1^{(\mathrm F)}=S/B=R_1$, together with
\eqref{eq:four-ratio-system}, gives the following.

\begin{lemma}[Exact matching through $R_4$]\label{lem:match-four}
For $0\leq j\leq4$,
\[
 R_j^{(\mathrm F)}=R_j.
\]
\end{lemma}

\begin{proof}
The assertion holds at $j=0,1$.  For the sequence in
\eqref{eq:free-model-ratios}, its three consecutive quotient
invariants are precisely
\[
 \frac{(R_{k+1}^{(\mathrm F)})^2}
 {R_k^{(\mathrm F)}R_{k+2}^{(\mathrm F)}}
 =\mathfrak q_k(A,B)\mathfrak q_k(C,D).
\]
Apply \eqref{eq:four-ratio-system} successively for $k=0,1,2$.
\end{proof}

Make the change $X=-y/S$ and put
\begin{align}
 P_{\mathrm F}(y)
 &:=
 \frac{J^{d,n}(-y/S)}{\gamma(n)},\label{eq:PF-def}\\
 p_{\mathrm F}(y)
 &:=
 \sum_{j=0}^d(-1)^j\binom dj
 \frac{(A)_j(C)_jD^j}
 {A^jC^j(B)_j(D)_j}y^j\notag\\
 &=
 {}_3F_2\left(
 \begin{matrix}-d,A,C\\ B,D\end{matrix};
 \frac{Dy}{AC}\right).
\label{eq:pF-def}
\end{align}
Define
\begin{equation}\label{eq:cF}
 c_j^{(\mathrm F)}
 :=
 \frac{R_j}{R_j^{(\mathrm F)}}.
\end{equation}
Then $P_{\mathrm F}$ is obtained from $p_{\mathrm F}$ by the
coefficient multiplier $c_j^{(\mathrm F)}$, and
\begin{equation}\label{eq:cF01234}
 c_0^{(\mathrm F)}=\cdots=c_4^{(\mathrm F)}=1.
\end{equation}

\section{Real-rootedness and localization of the comparison model}

The parameter construction has so far matched coefficients.  Its
second purpose is algebraic: after the change of variable
$X=-y/S$, the finite-free family becomes a coefficientwise
convolution of two Jacobi polynomials with positive zeros.

For polynomials written in the normalized form
\[
 p(y)=\sum_{j=0}^d(-1)^j\binom dj a_jy^j,\qquad
 q(y)=\sum_{j=0}^d(-1)^j\binom dj b_jy^j,
\]
their multiplicative finite-free convolution is
\[
 (p\boxtimes_d q)(y)
 :=\sum_{j=0}^d(-1)^j\binom dj a_jb_jy^j.
\]
This operation preserves positive real-rootedness
\cite[Proposition~2.7]{MMP}.  Its logarithmic-mesh property preserves
simplicity when one factor has distinct positive roots
\cite[Proposition~2.17]{MMP}.

\begin{lemma}[Multiplicative interval bound]
\label{lem:free-interval}
Suppose the roots of $p$ lie in $[u_-,u_+]\subset(0,\infty)$ and
the roots of $q$ lie in $[v_-,v_+]\subset(0,\infty)$.  Then every
root of $p\boxtimes_dq$ lies in
\[
 [u_-v_-,u_+v_+].
\]
\end{lemma}

\begin{proof}
For a polynomial with nonzero constant term, let
\[
 p^\vee(x):=\frac{x^dp(1/x)}{p(0)}.
\]
A coefficient calculation shows that
$(p\boxtimes_dq)^\vee=p^\vee\mathbin{\times_d}q^\vee$, where
$\times_d$ is the monic normalization of the same multiplicative
convolution.  The largest-root inequality
\[
 \lambda_{\max}(f\mathbin{\times_d}g)
 \leq\lambda_{\max}(f)\lambda_{\max}(g)
\]
is \cite[Theorem~1.13]{MSS}.  Applied to $p^\vee,q^\vee$, it gives
the lower endpoint $u_-v_-$.  Applying it once more after taking
reciprocal polynomials gives the upper endpoint $u_+v_+$.
\end{proof}

\begin{lemma}[Ordered multiplicative bound]
\label{lem:free-ordered}
Let
\[
 0<u_1\leq\cdots\leq u_d,\qquad
 0<w_1\leq\cdots\leq w_d
\]
be the ordered roots of $p$ and $p\boxtimes_dq$, respectively.  If
all roots of $q$ belong to $[v_-,v_+]\subset(0,\infty)$, then
\begin{equation}\label{eq:free-ordered-bound}
 v_-u_i\leq w_i\leq v_+u_i
 \qquad(1\leq i\leq d).
\end{equation}
\end{lemma}

\begin{proof}
We first recall a consequence of preservation of interlacing.  Give
positive-rooted monic polynomials the coordinatewise root order.  If
$f$ precedes $g$ in this order, one can pass from $f$ to $g$ by
moving the roots, one at a time, from the largest to the smallest.
Each consecutive pair in this finite chain interlaces.  Multiplicative
finite-free convolution with a positive-rooted polynomial preserves
the direction of interlacing \cite[Proposition~2.11]{MMP}; hence it
also preserves the coordinatewise root order.

Apply this observation to the reciprocal polynomials used in the
proof of Lemma~\ref{lem:free-interval}.  The roots of $q^\vee$ lie
in $[v_+^{-1},v_-^{-1}]$, so monotonicity and the identity
\[
 f\mathbin{\times_d}(x-a)^d=a^df(x/a)
\]
give
\[
 v_+^{-1}\lambda_i(p^\vee)
 \leq\lambda_i\bigl((p\boxtimes_dq)^\vee\bigr)
 \leq v_-^{-1}\lambda_i(p^\vee).
\]
Reciprocating and reversing the indices proves
\eqref{eq:free-ordered-bound}.
\end{proof}

\begin{lemma}[Finite-free model roots]\label{lem:free-roots}
Under \eqref{eq:free-positive-range}, the polynomial
$p_{\mathrm F}$ has $d$ distinct positive roots.  Every root and
every critical point $y$ satisfies
\begin{equation}\label{eq:free-root-location}
 |y-B|\leq C_{\mathrm{loc}}\sqrt{Bd},\qquad
 \frac14B\leq y\leq2B
\end{equation}
with an absolute constant $C_{\mathrm{loc}}$.
\end{lemma}

\begin{proof}
Factor \eqref{eq:pF-def} as
\begin{equation}\label{eq:free-factorization}
 p_{\mathrm F}=p_1\boxtimes_d p_2,
\end{equation}
where
\[
 p_1(y)={}_2F_1\left(-d,A;B;\frac yA\right),
 \qquad
 p_2(y)={}_2F_1\left(-d,C;D;\frac{Dy}{C}\right).
\]
For $U,V>0$, write
\[
 q_{U,V}(y):={}_2F_1\left(-d,U;V;\frac yU\right)
 =\frac{d!}{(V)_d}
 P_d^{(V-1,U-V-d)}\left(1-\frac{2y}{U}\right).
\]
Put $\alpha=V-1$, $\beta=U-V-d$, and
$H=\alpha+\beta=U-d-1$.  After transporting the standard Jacobi
matrix \cite[Chapter~IV]{Szego} by $y=U(1-t)/2$, its $k$th diagonal
entry is
\begin{equation}\label{eq:jacobi-diagonal}
 U\,\frac{H(V+2k)+2k(k+1)}
 {(H+2k)(H+2k+2)}.
\end{equation}
Its off-diagonal entry between rows $k-1$ and $k$ is
\begin{equation}\label{eq:jacobi-offdiagonal}
 \frac{U}{2k+H}
 \left\{
 \frac{k(k+\alpha)(k+\beta)(k+H)}
 {(2k+H-1)(2k+H+1)}
 \right\}^{1/2},\qquad 1\leq k<d.
\end{equation}
We record a ratio-free estimate for this matrix.  Suppose that
\begin{equation}\label{eq:Jacobi-ratio-free-range}
 U\geq V+d,\qquad V\geq32d.
\end{equation}
Then $H\geq V-1\geq31d$.  Subtracting $V$ from
\eqref{eq:jacobi-diagonal} gives the exact expression
\begin{equation}\label{eq:jacobi-diagonal-difference}
 \frac{
  2k(k+H+1)(U-2V)+VH(d-1)
 }{(H+2k)(H+2k+2)}.
\end{equation}
Since $|U-2V|\leq U$, $U=H+d+1$, and $V\leq H+1$,
the two terms in the absolute value of
\eqref{eq:jacobi-diagonal-difference} are bounded respectively by
\[
 2d\frac{32}{31}\frac{33}{31}
 \quad\hbox{and}\quad
 \frac{32}{31}d.
\]
Here we used $k+H+1\leq H+d\leq(32/31)H$,
$U\leq H+2d\leq(33/31)H$, and
$V\leq H+d\leq(32/31)H$.
Their sum is less than $4d$.  For an off-diagonal entry put
$\beta=U-V-d=H+1-V$.  Under
\eqref{eq:Jacobi-ratio-free-range}, $0\leq\beta\leq H$, and
\[
 \frac{(k+\beta)(k+H)}
 {(2k+H-1)(2k+H+1)}\leq1.
\]
Also $U/(2k+H)\leq33/31$ and
$k+V-1\leq V+d\leq(33/32)V$.  Hence every off-diagonal entry is at
most $2\sqrt{Vd}$, and the two entries adjacent to a row have sum
at most $4\sqrt{Vd}$.

Gershgorin's theorem therefore gives
\begin{equation}\label{eq:Jacobi-block-location}
 \operatorname{roots}(q_{U,V})
 \subset
 [V-C_{\mathrm J}\sqrt{Vd},\,V+C_{\mathrm J}\sqrt{Vd}]
\end{equation}
with, for example, $C_{\mathrm J}=8$.  Indeed,
$A\geq8B$, $B\geq K_{\mathrm r}d$, and
$K_{\mathrm r}\geq32$ imply
\eqref{eq:Jacobi-ratio-free-range} for $(U,V)=(A,B)$, so this
applies to $p_1=q_{A,B}$.

For the second factor, use
\[
 p_2(y)
 =\frac{d!}{(D)_d}
 P_d^{(D-1,C-D-d)}
 \left(1-\frac{2Dy}{C}\right).
\]
Both Jacobi parameters exceed $-1$.  The inequalities
$D\geq K_{\mathrm r}d$ and $C-D\geq4d$ imply
\eqref{eq:Jacobi-ratio-free-range} for $(U,V)=(C,D)$.  We need a
sharper estimate than \eqref{eq:Jacobi-block-location} for the
semicircle limit.  Set $G:=C-D$ and
$H=C-d-1=D+G-d-1$.  Under \eqref{eq:free-positive-range},
\[
 H\geq D,\qquad |C-2D|\leq D,\qquad
 \frac{C}{H}\leq\frac{17}{16}.
\]
The exact difference \eqref{eq:jacobi-diagonal-difference}, now with
$(U,V)=(C,D)$, is bounded in absolute value by $4d$.  Indeed,
$k+H+1\leq H+d\leq(33/32)H$, so its two numerator terms, after
division by the denominator, are at most $(66/32)d$ and $d$.

For the off-diagonal entries, $\beta=G-d$ and
\[
 k+D-1\leq\frac{33}{32}D,\qquad
 k+\beta\leq G,\qquad
 k+H\leq\frac{33}{32}H.
\]
Both denominator factors inside the square root in
\eqref{eq:jacobi-offdiagonal} are at least $H$.  It follows that
every off-diagonal entry is at most $2\sqrt{Gd}$.  Gershgorin's
theorem therefore gives
\[
 \operatorname{roots}(q_{C,D})
 \subset[D-4d-4\sqrt{Gd},\,D+4d+4\sqrt{Gd}].
\]
Since $p_2(y)=q_{C,D}(Dy)$,
\begin{equation}\label{eq:p2-sharp-location}
 \operatorname{roots}(p_2)\subset[1-\eta,1+\eta],
 \qquad
 \eta:=\frac{4(d+\sqrt{(C-D)d})}{D}<\frac12.
\end{equation}
The last inequality follows from $d/D\leq1/32$ and
$(C-D)/D\leq1/4$.
Both factors have simple positive roots.  The cited positivity and
logarithmic-mesh theorems give simplicity and positivity for
\eqref{eq:free-factorization}.  Lemma~\ref{lem:free-interval},
together with $D\asymp B$, gives
$|y-B|\leq C_{\mathrm{loc}}\sqrt{Bd}$ for every root.  Critical points lie
between consecutive roots.  Fixing $K_{\mathrm r}$ after
$C_{\mathrm{loc}}$ makes the last two inequalities in
\eqref{eq:free-root-location} follow from $B\geq K_{\mathrm r}d$.
\end{proof}

\section{The fifth-order residual multiplier}

We have constructed a positive-rooted model whose normalized
coefficients agree with those of the Jensen polynomial at
$j=0,1,2,3,4$.  We now continue the logarithm of their ratio away
from the integers and use those five zeros to control the entire
coefficient multiplier.

\begin{lemma}[Fifth-order residual bound]\label{lem:free-defect}
There are absolute constants $K_5,C_{\mathrm{def}}>0$ with the following
property.  Suppose $d\geq5$ and
\[
 n^3\log^2(n+2)\geq K_5d^5.
\]
Put
\[
 r_{\mathrm F}:=K_6\sqrt{Bd},\qquad
 \Omega_{\mathrm F}
 :=\{z\in\C:\dist(z,[0,d])\leq2r_{\mathrm F}\},
\]
where $K_6$ is a sufficiently large absolute constant.  There is
a holomorphic function $c_{\mathrm F}$ on a neighborhood of
$\Omega_{\mathrm F}$ such that
$c_{\mathrm F}(j)=c_j^{(\mathrm F)}$ for $0\leq j\leq d$ and
\begin{equation}\label{eq:cF-defect}
 \sup_{\Omega_{\mathrm F}}|c_{\mathrm F}-1|
 \leq
 C_{\mathrm{def}}\frac{d^{5/2}}{n^{3/2}\log(n+2)}.
\end{equation}
\end{lemma}

\begin{proof}
On the sector in Proposition~\ref{prop:sectorial-saddle}, let
$g=\log\gamma$ be the branch real on the positive axis.  Define
\begin{align}
 E_{\mathrm F}(z)
 &:=
 g(n+z)-g(n)-z\log S+z\log A+z\log C-z\log D
 \notag\\
 &\quad
 +\log\Gamma(B+z)-\log\Gamma(B)
 -\log\Gamma(A+z)+\log\Gamma(A)\notag\\
 &\quad
 +\log\Gamma(D+z)-\log\Gamma(D)
 -\log\Gamma(C+z)+\log\Gamma(C).
\label{eq:EF-def}
\end{align}
At every integer $0\leq j\leq d$,
$e^{E_{\mathrm F}(j)}=c_j^{(\mathrm F)}$.

The fifth-wedge hypothesis and
Lemma~\ref{lem:four-ratio-parameters} imply
\[
 A\asymp n\log n,\quad
 B\asymp C\asymp D\asymp n,\quad
 d+2r_{\mathrm F}=o(n).
\]
After increasing $K_5$, the set $n+\Omega_{\mathrm F}$ lies in the
sectorial neighborhood on which $g$ is defined, and every term in
\eqref{eq:EF-def} is holomorphic on a neighborhood of
$\Omega_{\mathrm F}$.

Let $b=n+\frac12$.  Five differentiations, followed by
\eqref{eq:gamma-h}, give
\begin{align}
 E_{\mathrm F}^{(5)}(z)
 &=
 h^{(5)}(n+z)
 +\psi^{(4)}(B+z)-\psi^{(4)}(b+z)-\psi^{(4)}(A+z)\notag\\
 &\quad+\psi^{(4)}(D+z)-\psi^{(4)}(C+z).
\label{eq:EF-fifth}
\end{align}
Equation \eqref{eq:hfifth} bounds the first term by
$C_\ast/(n^4\log n)$.  The parameter asymptotics give
\[
 |B-b|+|C-D|\leq\frac{C_\ast n}{\log n}.
\]
The series for $\psi^{(5)}$ and the mean-value formula therefore give
\[
 |\psi^{(4)}(B+z)-\psi^{(4)}(b+z)|
 +|\psi^{(4)}(D+z)-\psi^{(4)}(C+z)|
 \leq\frac{C_\ast}{n^4\log n}.
\]
Finally,
$|\psi^{(4)}(A+z)|\leq C_\ast/A^4$.
Consequently
\begin{equation}\label{eq:EF-fifth-bound}
 \sup_{\Omega_{\mathrm F}}|E_{\mathrm F}^{(5)}|
 \leq\frac{C_\ast}{n^4\log(n+2)}.
\end{equation}

By \eqref{eq:cF01234}, the real value of $E_{\mathrm F}$ vanishes
at $0,1,2,3,4$.  The set $\Omega_{\mathrm F}$ is convex, so the
Hermite--Genocchi formula and \eqref{eq:EF-fifth-bound} give
\[
 |E_{\mathrm F}(z)|
 \leq
 \frac{C_\ast}{5!\,n^4\log(n+2)}
 |z(z-1)(z-2)(z-3)(z-4)|.
\]
Since $r_{\mathrm F}\geq d$ and $B\asymp n$,
\[
 \sup_{\Omega_{\mathrm F}}|E_{\mathrm F}|
 \leq
 C_\ast\frac{r_{\mathrm F}^5}{n^4\log(n+2)}
 \leq
 C_\ast\frac{d^{5/2}}{n^{3/2}\log(n+2)}.
\]
For $K_5$ large the last expression is at most $1/2$.
Taking $c_{\mathrm F}=e^{E_{\mathrm F}}$ and using
$|e^w-1|\leq2|w|$ for $|w|\leq1/2$ proves
\eqref{eq:cF-defect}.
\end{proof}

\section{Critical-point derivative bounds}

The multiplier estimate in Lemma~\ref{lem:free-defect} verifies the
analytic hypothesis of Proposition~\ref{prop:abstract-stability}.  It
remains to verify the derivative-ratio hypothesis at the critical
points of the finite-free model.

\begin{lemma}[Critical-point derivative ratios]\label{lem:free-ratios}
Let the parameters be the solution supplied by
Lemma~\ref{lem:four-ratio-parameters}, and assume
\eqref{eq:fifth-wedge-hypothesis} with its constant sufficiently
large.  Let $y$ be a critical point of $p_{\mathrm F}$ and put
\[
 T_k:=\frac{y^kp_{\mathrm F}^{(k)}(y)}
 {p_{\mathrm F}(y)}.
\]
If $K_6$ in the definition of $r_{\mathrm F}$ is sufficiently
large, then, for all sufficiently large $n$,
\begin{equation}\label{eq:free-Tk}
 |T_k|\leq r_{\mathrm F}^k
 \qquad(0\leq k\leq d).
\end{equation}
\end{lemma}

\begin{proof}
The hypergeometric equation is a small perturbation of the Jacobi
equation because $(C-D)/C=O(\mathcal L_n^{-1})$.  We first control
the failure of $p_{\mathrm F}$ to satisfy the Jacobi equation and
then use the differentiated Jacobi recurrence to bound all the
ratios $T_k$.

Write $\mathcal E=y\,\dd/\dd y$ for the Euler operator and
\[
 \mathcal Jp
 :=
 y\left(1-\frac yA\right)p''
 +\left\{B-y+\frac{(d-1)y}{A}\right\}p'
 +dp.
\]
The hypergeometric equation for \eqref{eq:pF-def} is
\[
 \left[
 \mathcal E(\mathcal E+B-1)(\mathcal E+D-1)
 -\frac{Dy}{AC}(\mathcal E-d)(\mathcal E+A)(\mathcal E+C)
 \right]p_{\mathrm F}=0.
\]
Put $\varepsilon_{\mathrm p}=(C-D)/C$.  Since
$D/C=1-\varepsilon_{\mathrm p}$ and
$(\mathcal E+D-1)y=y(\mathcal E+D)$, the equation factors exactly as
\begin{equation}\label{eq:perturbed-Jacobi-factor}
 (\mathcal E+D)\mathcal Jp_{\mathrm F}
 =
 -\frac{\varepsilon_{\mathrm p}}{A}
 \mathcal E(\mathcal E-d)(\mathcal E+A)p_{\mathrm F}.
\end{equation}

Set $W=\mathcal Jp_{\mathrm F}$ and, at the fixed critical point,
\[
 U_m:=\frac{y^mW^{(m)}(y)}{p_{\mathrm F}(y)}.
\]
Differentiating \eqref{eq:perturbed-Jacobi-factor} $m$ times gives
\begin{equation}\label{eq:U-backward}
 U_{m+1}+(D+m)U_m=-\varepsilon_{\mathrm p}V_m,
\end{equation}
where
\[
 V_m
 :=
 \frac{y^m}{A\,p_{\mathrm F}(y)}
 \left\{\mathcal E(\mathcal E-d)(\mathcal E+A)
 p_{\mathrm F}\right\}^{(m)}(y).
\]
Since $W$ has degree at most $d$, $U_{d+1}=0$.

Write $r=r_{\mathrm F}$ and let
\[
 \mathcal M:=\max_{0\leq k\leq d}\frac{|T_k|}{r^k}.
\]
Expanding the cubic Euler operator and differentiating monomials
gives, with $T_k=0$ for $k>d$,
\[
 V_m=\frac1A\{T_{m+3}
 +(A-d+3+3m)T_{m+2}
 +\beta_mT_{m+1}+\gamma_mT_m\},
\]
where the coefficients are explicitly
\[
 \beta_m
 =A(2m+1-d)-d(2m+1)+3m^2+3m+1,
 \qquad
 \gamma_m=m(m-d)(m+A).
\]
For $0\leq m\leq d$, the bounds in
\eqref{eq:free-positive-range} imply
$|\beta_m|\leq C_\ast Ad$ and
$|\gamma_m|\leq C_\ast Ad^2$.  Consequently
\begin{equation}\label{eq:Vm-quantified}
 |V_m|\leq C_\ast\mathcal M r^{m+2}
 \left(1+\frac rA+\frac dA+\frac dr+\frac{d^2}{r^2}\right).
\end{equation}
If $r/D\leq1/4$, backward substitution in
\eqref{eq:U-backward}, beginning with $U_{d+1}=0$, gives
\begin{equation}\label{eq:U-bound}
 |U_m|
 \leq
 \frac{C_\ast\varepsilon_{\mathrm p}}{D(1-r/D)}\,
 \mathcal M r^{m+2}
 \left(1+\frac rA+\frac dA+\frac dr+\frac{d^2}{r^2}\right).
\end{equation}

On the other hand, differentiating the Jacobi operator gives the
exact identity
\begin{align}
 yU_m
 &=
 \left(1-\frac yA\right)T_{m+2}
 +\left\{B+m-y+\frac{(d-1-2m)y}{A}\right\}T_{m+1}\notag\\
 &\quad
 +(d-m)\left(1+\frac mA\right)yT_m.
\label{eq:free-Jacobi-recurrence}
\end{align}
Since $y\leq2B$ and $A\geq8B$, we have
$1-y/A\geq3/4$.  For $0\leq m\leq d-2$, define
\begin{align*}
 a_m&:=\frac{
 \left|B+m-y+(d-1-2m)y/A\right|}
 {(1-y/A)r},\\
 b_m&:=\frac{(d-m)(1+m/A)y}{(1-y/A)r^2}.
\end{align*}
Lemma~\ref{lem:free-roots} and $r=K_6\sqrt{Bd}$ give the explicit
bounds
\begin{align}
 a_m&\leq\frac43\left(
  \frac{C_{\mathrm{loc}}}{K_6}+\frac dr
  +\frac{6dB}{Ar}\right),\label{eq:am-bound}\\
 b_m&\leq\frac{8}{3K_6^2}\left(1+\frac dA\right).
\label{eq:bm-bound}
\end{align}
The contribution of $yU_m$ in
\eqref{eq:free-Jacobi-recurrence}, after division by
$(1-y/A)r^{m+2}$, is at most $\vartheta_{n,d}\mathcal M$, where
\begin{equation}\label{eq:vartheta}
 \vartheta_{n,d}:=
 \frac{4C_\ast\varepsilon_{\mathrm p}}{3}
 \frac{y/D}{1-r/D}
 \left(1+\frac rA+\frac dA+\frac dr+\frac{d^2}{r^2}\right).
\end{equation}

The wedge hypothesis implies
\[
 \frac dn\leq K_4^{-1/5}n^{-2/5}\log^{2/5}(n+2).
\]
Together with the parameter asymptotics, this shows uniformly in
the wedge that
\[
 \frac rA,\ \frac dA,\ \frac dr,\ \frac{d^2}{r^2},\
 \frac rD\longrightarrow0,\qquad
 \varepsilon_{\mathrm p}=O(\mathcal L_n^{-1}),\qquad
 \frac yD=O(1).
\]
Choose $K_6$ so that the fixed terms in
\eqref{eq:am-bound}--\eqref{eq:bm-bound} are small, and then choose
$n$ large enough that
\[
 a_m\leq\frac14,\qquad b_m\leq\frac14,
 \qquad \vartheta_{n,d}\leq\frac14
\]
for every admissible $d$ and $m$.  Equations \eqref{eq:U-bound} and
\eqref{eq:free-Jacobi-recurrence} now show
\[
 \frac{|T_{m+2}|}{r^{m+2}}
 \leq
 \frac14\frac{|T_{m+1}|}{r^{m+1}}
 +\frac14\frac{|T_m|}{r^m}
 +\frac14\mathcal M.
\]
Now $T_0=1$ and $T_1=0$.  If $\mathcal M>1$, choose an index at
which the maximum is attained.  It is at least $2$, and the last
inequality gives $\mathcal M\leq3\mathcal M/4$, a contradiction.
Thus
$\mathcal M\leq1$, proving \eqref{eq:free-Tk}.
\end{proof}

\section{Completion of the proof}

\begin{proof}[Proof of Theorem~\ref{thm:main}]
Let $n_0$ exceed all fixed thresholds in the preceding lemmas.
Choose $K$ larger than $K_4,K_5$, the constants required in the
root and derivative-ratio estimates, and
\[
 1+\max_{0\leq m<n_0}m^3\log^2(m+2).
\]
Increase it once more so that \eqref{eq:cF-defect} is at most
$1/2$ whenever the main hypothesis holds.  Then
$n^3\log^2(n+2)\geq Kd^5$ implies $n\geq n_0$.

If $d\leq4$, then
Lemma~\ref{lem:match-four} matches every coefficient of the two
degree-$d$ polynomials, so $P_{\mathrm F}=p_{\mathrm F}$.
Lemma~\ref{lem:free-roots} proves the conclusion directly.

Suppose $d\geq5$.
Lemma~\ref{lem:free-roots} gives $d$ simple positive zeros for
$p_{\mathrm F}$.  Lemma~\ref{lem:free-defect} supplies a holomorphic
multiplier $c_{\mathrm F}$ on $\Omega_{\mathrm F}$ with
\[
 \sup_{\Omega_{\mathrm F}}|c_{\mathrm F}-1|\leq\frac12.
\]
At the integers,
$c_{\mathrm F}(j)=R_j/R_j^{(\mathrm F)}>0$, and
\eqref{eq:cF01234} gives
$c_{\mathrm F}(0)=\cdots=c_{\mathrm F}(4)=1$.
Lemma~\ref{lem:free-ratios} verifies
\eqref{eq:abstract-derivative-ratios} with $r=r_{\mathrm F}$.
All the hypotheses of Proposition~\ref{prop:abstract-stability} are
therefore satisfied, so $P_{\mathrm F}$ has $d$ simple positive
zeros.

Finally,
\[
 P_{\mathrm F}(y)
 =\frac{J^{d,n}(-y/S)}{\gamma(n)}
\]
with $S>0$.  Hence $J^{d,n}$ has $d$ distinct negative real zeros.
\end{proof}

\section{The joint semicircle limit}
\label{sec:joint-semicircle}

Let $\mu_{\mathrm{sc}}$ denote the probability measure
\[
 \dd\mu_{\mathrm{sc}}(x)
 :=\frac1{2\pi}\sqrt{4-x^2}\,
   \boldsymbol 1_{[-2,2]}(x)\,\dd x.
\]

\begin{lemma}[Semicircle law for the comparison model]
\label{lem:model-semicircle}
Suppose that $n,d\to\infty$ through pairs satisfying the hypothesis
of Theorem~\ref{thm:main}.  If
$0<y_1\leq\cdots\leq y_d$ are the roots of $p_{\mathrm F}$, then
\begin{equation}\label{eq:model-semicircle}
 \frac1d\sum_{i=1}^d
 \delta_{(y_i-B)/\sqrt{Bd}}
 \Longrightarrow\ \mu_{\mathrm{sc}}.
\end{equation}
\end{lemma}

\begin{proof}
The parameter asymptotics in
Lemmas~\ref{lem:Jacobi-refined} and
\ref{lem:four-ratio-parameters} give, with $G:=C-D$,
\begin{equation}\label{eq:semicircle-parameter-limits}
 \frac Bn\longrightarrow1,\qquad
 \frac Dn\longrightarrow\frac12,\qquad
 \frac BA\longrightarrow0,\qquad
 \frac GB\longrightarrow0.
\end{equation}
The wedge gives $d/B\to0$.

First consider the roots $u_1\leq\cdots\leq u_d$ of
$p_1=q_{A,B}$.  Let $\mathcal J_1$ be the Jacobi matrix with entries
\eqref{eq:jacobi-diagonal} and
\eqref{eq:jacobi-offdiagonal}, for $(U,V)=(A,B)$, and write its
diagonal and off-diagonal entries as $D_k^{(1)}$ and $E_k^{(1)}$.
The estimate proved before \eqref{eq:Jacobi-block-location} gives
\begin{equation}\label{eq:semicircle-diagonal}
 \max_{0\leq k<d}
 \frac{|D_k^{(1)}-B|}{\sqrt{Bd}}
 \leq4\sqrt{\frac dB}=o(1).
\end{equation}
If $H=A-d-1$, the exact off-diagonal formula gives, uniformly for
$1\leq k<d$,
\begin{align}
 \frac{(E_k^{(1)})^2}{Bk}
 &=\frac{A^2}{(H+2k)^2}\frac{B+k-1}{B}
   \frac{(H+k+1-B)(H+k)}
   {(H+2k-1)(H+2k+1)}\notag\\
 &=1+O\!\left(\frac dB+\frac BA\right)=1+o(1).
\label{eq:semicircle-offdiagonal}
\end{align}
Here $A/(H+2k)=1+O(d/A)$,
$(B+k-1)/B=1+O(d/B)$, and the final quotient on the first line is
$1+O(B/A+d/A)$, all uniformly in $k$.

Let $\mathcal H_d$ be the $d\times d$ tridiagonal matrix with zero
diagonal and off-diagonal entries $\sqrt{k}$ between rows $k-1$ and
$k$, $1\leq k<d$.  Equations
\eqref{eq:semicircle-diagonal} and
\eqref{eq:semicircle-offdiagonal} imply
\begin{equation}\label{eq:semicircle-matrix-approximation}
 \left\|
  \frac{\mathcal J_1-BI}{\sqrt{Bd}}
 -\frac{\mathcal H_d}{\sqrt d}
 \right\|_{\mathrm{op}}=o(1).
\end{equation}
We used here the elementary bound that the norm of a symmetric
tridiagonal matrix is at most the largest absolute diagonal entry
plus twice the largest absolute off-diagonal entry.
The characteristic polynomial of $\mathcal H_d$ is the monic
probabilists' Hermite polynomial.  Its normalized empirical spectral
measure tends to $\mu_{\mathrm{sc}}$.  For completeness, this follows
directly from traces: a nearest-neighbor walk count gives, for every
fixed $m\geq0$,
\[
 \lim_{d\to\infty}\frac1d
 \operatorname{tr}\left(\frac{\mathcal H_d}{\sqrt d}\right)^{2m}
 =\frac1{m+1}\binom{2m}{m},
 \qquad
 \operatorname{tr}(\mathcal H_d^{\,2m+1})=0.
\]
Indeed, away from the two boundary rows, each closed walk of length
$2m$ has $m$ upward and $m$ downward steps and contributes
$(k/d)^m+O_m(d^{-1})$ when it starts in row $k$; averaging over $k$
produces
$\binom{2m}{m}\int_0^1x^m\dd x$.  The matrices
$\mathcal H_d/\sqrt d$ have norm at most $2$, so the moment limits
give weak convergence.  Weyl's inequality and
\eqref{eq:semicircle-matrix-approximation} now show that
\begin{equation}\label{eq:p1-semicircle}
 \frac1d\sum_{i=1}^d
 \delta_{(u_i-B)/\sqrt{Bd}}
 \ \Longrightarrow\ \mu_{\mathrm{sc}}.
\end{equation}

It remains to check that the second finite-free factor is negligible
on this scale.  By \eqref{eq:p2-sharp-location} and
Lemma~\ref{lem:free-ordered},
\[
 |y_i-u_i|\leq\eta u_i,\qquad
 \eta=\frac{4(d+\sqrt{Gd})}{D}.
\]
Equation \eqref{eq:Jacobi-block-location} gives
$u_i\leq B+C_{\mathrm J}\sqrt{Bd}$.  Since $B/D$ stays bounded,
\begin{align*}
 \max_{1\leq i\leq d}\frac{|y_i-u_i|}{\sqrt{Bd}}
 &\ll
 \frac{d+\sqrt{Gd}}D
 \frac{B+\sqrt{Bd}}{\sqrt{Bd}}\\
 &\ll
 \left(\sqrt{\frac dB}+\sqrt{\frac GB}\right)
 \left(1+\sqrt{\frac dB}\right)=o(1)
\end{align*}
by \eqref{eq:semicircle-parameter-limits}.  Combining this uniform
pairing with \eqref{eq:p1-semicircle} proves
\eqref{eq:model-semicircle}.
\end{proof}

\begin{proof}[Proof of Theorem~\ref{thm:semicircle}]
Suppress the index $\nu$.  Let
$0<\widehat y_1\leq\cdots\leq\widehat y_d$ be the roots of
$P_{\mathrm F}$.  The proof of Theorem~\ref{thm:main} applies
Proposition~\ref{prop:abstract-stability} to
$p_{\mathrm F}$ and $P_{\mathrm F}$.  Consequently,
\eqref{eq:abstract-counting-transfer} gives
\begin{equation}\label{eq:Jensen-model-Kolmogorov}
 \sup_{t>0}\left|
 \frac1d\#\{i:\widehat y_i\leq t\}
 -\frac1d\#\{i:y_i\leq t\}
 \right|\leq\frac1d.
\end{equation}
The same bound holds after any one-to-one affine change of variable.

If $\rho_i$ is the Jensen zero corresponding to $\widehat y_i$,
then $S=BR_1$ and \eqref{eq:PF-def} give
\[
 \rho_i=-\frac{\widehat y_i}{BR_1},
 \qquad
 \sqrt{\frac nd}(1+R_1\rho_i)
 =-\sqrt{\frac nB}\,
   \frac{\widehat y_i-B}{\sqrt{Bd}}.
\]
Lemma~\ref{lem:model-semicircle},
\eqref{eq:Jensen-model-Kolmogorov}, and $B/n\to1$ show that the
empirical measure of the expressions on the left converges to the
reflection of $\mu_{\mathrm{sc}}$.  The semicircle law is symmetric,
so this reflection is $\mu_{\mathrm{sc}}$ itself.
\end{proof}

\section{Discussion}

The published theorem of Griffin, Ono, Rolen, Thorner, Tripp, and
Wagner gives hyperbolicity for $n\geq ce^d$; arXiv v3 records the
sharper threshold $n\geq ce^{d/2}$
\cite[Theorem~1.1]{GORTTW}.  Their comparison is Hermite.  Here the
comparison stays in the positive-root class: the positive error at
$R_3$ selects the first Jacobi deformation, and the negative error
at $R_4$ is absorbed by the second finite-free factor.

Theorem~\ref{thm:semicircle} improves the order of limits in the
global derivative-aspect GUE statement of \cite{GORZ}.  In their
notation $\delta(n)\sim(2n)^{-1/2}$, so the factor
$\sqrt{n/d}$ in Theorem~\ref{thm:semicircle} is the leading form of
$1/(\sqrt{2d}\,\delta(n))$; centering at $-1/R_1$ uses the exact
first coefficient ratio.  Their theorem first fixes $d$ and lets
$n\to\infty$, obtaining a Hermite polynomial, after which the
semicircle law is recovered as $d\to\infty$.  Here the two limits
are simultaneous throughout the polynomial wedge.  No claim is made
about local GUE correlations.

After $R_1,\ldots,R_4$ have been matched, the logarithmic multiplier
vanishes at five consecutive indices and has fifth derivative
$O(1/(n^4\log n))$.  The root localization and perturbed Jacobi
equation give the derivative scale
$r_{\mathrm F}\asymp\sqrt{nd}$.  Consequently the remaining error is
\[
 O\!\left(\frac{(\sqrt{nd})^5}{n^4\log n}\right)
 =O\!\left(\frac{d^{5/2}}{n^{3/2}\log n}\right),
\]
which yields the stated wedge.  The constants in the Jacobi
localization can be taken explicitly; the proof above uses
$C_{\mathrm J}=8$, and the localization itself requires only
$A\geq B+d$ rather than $A\geq64(B+d)$.  The later differential
recurrence is where the modest condition $A\geq8B$ is used.  No
attempt is made to optimize the eventual absolute constant.

This result controls an asymptotic region and does not provide a
converse route from partial Jensen hyperbolicity to the Riemann
hypothesis; see Farmer \cite{Farmer}.  O'Sullivan's modified
P\'olya--Jensen criterion gives a complementary use of Hermite
combinations \cite{OSullivan}.

\end{document}